\pdfoutput=1
\RequirePackage{ifpdf}
\ifpdf 
\documentclass[pdftex]{sigma}
\else
\documentclass{sigma}
\fi

\numberwithin{equation}{section}

\newtheorem{Theorem}{Theorem}[section]
\newtheorem{Corollary}[Theorem]{Corollary}
\newtheorem{Lemma}[Theorem]{Lemma}
\newtheorem{Proposition}[Theorem]{Proposition}
 { \theoremstyle{definition}
\newtheorem{Definition}[Theorem]{Definition}

\newtheorem{Example}[Theorem]{Example}
\newtheorem{Remark}[Theorem]{Remark}}

\newcommand{\N}{\mathbb N}
\newcommand{\Z}{\mathbb Z}
\newcommand{\F}{\mathbb F}

\newcommand{\heap}{\operatorname{Heap}}
\newcommand{\im}{\operatorname{Im}}
\newcommand{\id}{\mathrm{id}}
\def\nij#1{{\underset{#1}{\circ}}}

\begin{document}
\allowdisplaybreaks

\newcommand{\arXivNumber}{2303.12880}

\renewcommand{\PaperNumber}{056}

\FirstPageHeading

\ShortArticleName{Affine Nijenhuis Operators and Hochschild Cohomology of Trusses}

\ArticleName{Affine Nijenhuis Operators\\ and Hochschild Cohomology of Trusses}

\Author{Tomasz BRZEZI\'NSKI~$^{\rm ab}$ and James PAPWORTH~$^{\rm a}$}

\AuthorNameForHeading{T.~Brzezi\'nski and J.~Papworth}

\Address{$^{\rm a)}$~Department of Mathematics, Swansea University, Fabian Way, Swansea SA1 8EN, UK}
\EmailD{\href{mailto:t.brzezinski@swansea.ac.uk}{t.brzezinski@swansea.ac.uk}, \href{mailto:j.papworth.918550@swansea.ac.uk}{j.papworth.918550@swansea.ac.uk}}

\Address{$^{\rm b)}$~Faculty of Mathematics, University of Bia{\l}ystok, K.~Cio{\l}kowskiego~1M, \\
\hphantom{$^{\rm b)}$}~15-245 Bia{\l}ystok, Poland}

\ArticleDates{Received April 04, 2023, in final form July 27, 2023; Published online August 04, 2023}

\Abstract{The classical Hochschild cohomology theory of rings is extended to abelian heaps with distributing multiplication or trusses. This cohomology is then employed to give necessary and sufficient conditions for a Nijenhuis product on a truss (defined by the extension of the Nijenhuis product on an associative ring introduced by Cari\~nena, Grabowski and Marmo in [\textit{Internat.~J.~Modern Phys.~A} \textbf{15} (2000), 4797--4810, arXiv:math-ph/0610011]) to be associative. The definition of Nijenhuis product and operators on trusses is then linearised to the case of affine spaces with compatible associative multiplications or associative {\em affgebras}. It is shown that this construction leads to compatible Lie brackets on an affine space.}

\Keywords{Nijenhuis operator; Hochschild cohomology; truss; heap; affine space}

\Classification{20N10; 16E40; 81R12}

\section{Introduction}

As shown by Magri \cite{Mag:sim} two compatible Poisson structures are closely related to integrability of classical Hamiltonian systems. Extending this idea to quantum mechanics, Cari\~nena, Grabowski and Marmo \cite{4} proposed a way of deforming a given product on an algebra (of operators on a~Hilbert space) so that two compatible Lie algebra structures are obtained. This deformation involves an operator $N$ acting on an associative algebra $A$ and satisfying the following simple equation:
\begin{gather}\label{int.Nij}
 N(a)N(b) = N(N(a)b-N(ab)+aN(b)) \qquad \mbox{for all}\ a,b\in A.
\end{gather}
Borrowing terminology from differential geometry and Lie algebra theory, $N$ is called a {\em Nijenhuis tensor} in \cite{4}. The combination of signs on the right-hand side of \eqref{int.Nij} indicates immediately that the Nijenhuis condition has an {\em affine} rather than {\em linear} flavour. The aim of this paper is to demonstrate how one may extend Nijenhuis tensors or operators to affine spaces with compatible associative multiplications. This in turn might allow one to develop the gauge or frame-independent theory of quantum bi-Hamiltonian systems in the spirit of \cite{Ben:fib, MasVig:non, Tul:fra, Urb:aff, Wei:uni}.

Affine spaces admit a natural ternary operation and thus can be interpreted as \textit{heaps} (see Definition~\ref{heap}). Bi-affine multiplication on an affine space distributes over the ternary heap operation, and so just as any associative algebra is a ring, an associative `affgebra' is a \textit{truss} (see Definition~\ref{truss}). Thus we study deformations of products on trusses through heap operators and the resulting Nijenhuis conditions before showing how we may think of affine Nijenhuis operators on an affine space with a compatible associative multiplication.

The condition \eqref{int.Nij} is sufficient but not necessary for the associativity of the deformed product. In view of the classical results of Gerstenhaber on deformations of rings \cite{7} it is not entirely surprising that, as the authors of \cite{4} observe, the associativity of the deformed product is fully controlled by two-cocycles in the Hochschild cohomology of $A$ with coefficients in $A$ \cite{Hoch}. As the same can be expected of the deformed products of trusses, and indeed it is the case, this leads us to developing rudiments of the Hochschild cohomology for trusses. The main difficulty here is that the category of trusses is not enriched over the category of abelian groups but over the category of abelian heaps. The latter has no zero object and thus the usual methods of homological algebra cannot be applied. To overcome this difficulty, we take
any element $e$ of a~truss~$T$, retract the heap underlying $T$ to an abelian group and build a cochain complex of bi-heap homomorphisms $T^n\to T$ in that way. Extra care should be taken due to the facts that first it is not guaranteed that $e$ `behaves like the zero', i.e., the product $ae$ is not necessarily $e$, second that $e$ might not be preserved by the cochains (multi-heap homomorphisms), and third that we would like the coboundary operators to preserve constant functions with value~$e$, so that they are homomorphisms of corresponding abelian groups (retracts at the constant cochains with value~$e$). The construction of the \textit{$e$-relative Hochschild cochain complex} is achieved in Definition~\ref{def.cob} and Theorem~\ref{deltathme}. The corresponding cohomology is defined in Definitions~\ref{def.cob.coc} and~\ref{quotient.def} and it is shown to be independent from the choice of base elements (up to isomorphism) in Theorem~\ref{thm.Hoch}.

After the reworking of Hochschild cohomology we proceed to define a \textit{Nijenhuis product} on a~truss as a deformation of the original multiplication by a heap endomorphism $N$ combined with the ternary heap operation (see Definition~\ref{def.Nij.tor}). This mimics the construction in \cite{4} (and, of course, reduces to it in the additive case). We show in Theorem~\ref{2nd} that the Nijenhuis product on $T$ given by $N$ is associative if and only if, for all $e\in T$ (equivalently, for any $e\in T$), the \textit{$e$-Nijenhuis torsion} of $N$ introduced in Definition~\ref{def.Nij.tor} is a 2-cocycle in the $e$-relative Hoschschild cohomology of $T$. The operator $N$ that is a homomorphism between the truss with the new (deformed) and original products is termed a \textit{Nijenhuis operator}. This is equivalent to the triviality of its $e$-torsion for any $e\in T$.

The remainder of Section~\ref{section4} focuses on examples and properties of Nijenhuis operators.
In particular, in Proposition~\ref{prop.Nij.z} we classify all Nijenhuis operators on commutative trusses built on the abelian group of integers. Then we develop iterative procedure to construct Nijenhuis operators in Theorem~\ref{thm.Nij.iter} and study the compatibility of Nijenhuis operators, showing that powers of Nijenhuis operators are pairwise-compatible (see Definition~\ref{def.compat} and Theorem~\ref{thm.comp}).

Section~\ref{five} deals with the extension of results found in \cite{4} from linear maps to affine maps, creating an affine version of quantum bi-Hamiltonian systems and \textit{affine Nijenhuis operators}.
In particular we discusses the compatibility conditions and examples of affine Nijenhuis operators, showing in particular that barycentric combinations of affine Nijenhuis operators on an associative algebra or, more generally, affgebra form Nijenhuis operators (see Theorem~\ref{thm.affine} and Proposition~\ref{prop.affine}). The paper ends with Theorem~\ref{thm.affine.Lie}, which may be interpreted as an affine version of a (weak) quantum bi-Hamiltonian system, since it shows how an affine Nijenhuis operator induces a Lie bracket on an affine space compatible with the Lie bracket given by the commutators. This is an affine version of \cite[Theorem~8]{4}.

\section{Preliminaries on heaps and trusses}
The following will get the reader up to speed on prerequisite knowledge of both heaps \cite{5, Brz:par, 2, 6} and trusses \cite{3, Brz:par} for the later sections of this paper.
\begin{Definition}\label{heap}
A \textit{heap} is an algebraic structure $(H, [---])$, where $H$ is a set and $[---]$ is a ternary operation
\[[---]\colon \ H^{3} \rightarrow H, \qquad (a,b,c) \mapsto [a,b,c],\]
such that for all $a,b,c,d,e \in H$ we have the following properties:
\begin{align*}
& \text{associativity:} \quad [[a,b,c],d,e] = [a,b,[c,d,e]],\\
& \text{Mal'cev identities:} \quad [a,a,b] = b = [b,a,a].
\end{align*}

Furthermore, a heap is said to be \textit{abelian} if, for all $a,b,c \in H$,
\[[a,b,c]=[c,b,a].\]
\end{Definition}
\begin{Remark}\label{rem.rules}
In an abelian heap $(H,[---])$, the placement of brackets in multiple application of the heap operation does not play any role, hence we will write $[a_1,a_2,\dots ,a_{2n+1}]$ for any such multiple application. However, the parity of the position of an element does matter, any element within an even or odd position in the operation may exchange position with any other element in a respectively even or odd position. Moreover, if after such a parity preserving rearrangement two adjacent elements are equal to each other, we may perform a cancellation of these elements. For example,
\[[a,x,b,c,x] = [a,c,b,x,x] = [a,c,b] \]
for any $a,b,c,x \in H$.
\end{Remark}
\begin{Remark}\label{rem.equal}
The axioms of a heap $H$ imply in particular that any three elements in the expression $[a,b,c]=d$ determine the fourth one. In particular, $[a,b,c]=d$ if and only if ${[c,d,a]=b}$. Furthermore, $a=b$ if and only if $[a,b,c]=c$ if and only if $[c,a,b]= c$, for all, equivalently any~$c\in H$. The transition from any to all is clear from the chain of arguments $[a,b,c]=c$ if and only if $[[a,b,c],c,d]=[c,c,d]$, if and only if $[a,b,d]=d$ by the heap associativity and Mal'cev identities.
\end{Remark}
\begin{Definition}
A \textit{heap homomorphism} or a \textit{heap map} from $(H,[---])$ to $(H',[---])$ is a~mapping $f\colon H \rightarrow H'$ that preserves the ternary operation, that is, for all $a,b,c \in H$, \[f([a,b,c])=[f(a),f(b),f(c)].\]
The set of all heap homomorphisms from $(H,[---])$ to $(H',[---])$ is denoted by $\heap(H,H')$.
\end{Definition}
The set $\heap(H,H')$ includes in particular all constant functions.
In case $H$, $H'$ are abelian, $\heap(H,H')$ is a heap by a pointwise operation, $[f,g,h](a) = [f(a),g(a),h(a)]$. Any singleton set is an (abelian) heap with a trivial (only possible) operation. The unique function from any heap to the singleton heap is a heap homomorphism. This makes the singleton set a terminal object in the category of heaps; we denote it by $\ast$.

Similarly, the empty set is an abelian heap, the initial object in the category of heaps. To maintain the correspondence between heaps and groups described in the following remark and Definition~\ref{ret} we assume throughout that the discussed heaps are {\em non-empty}.
\begin{Remark}
Given a group $(G,\cdot,1)$, we define a ternary operation
\[ [---]\colon\ G^3 \rightarrow G, \qquad [x,y,z] = x \cdot y^{-1} \cdot z.\]
 Then $(G,[---])$ is a heap denoted by $H(G)$. If $(G,\cdot,1)$ is abelian, then so is $H(G)$.
\end{Remark}
\begin{Definition} \label{ret}
Let $(H,[---])$ be a heap and $e \in H$. Define the following binary operation on~$H$:
\[\cdot_e \colon\ H \times H \rightarrow H, \qquad x \cdot_e y = [x,e,y]. \]
 Then $G(H;e)$ is a group known as the \textit{retract} of $H$. Finally, note that if $(H,[---])$ is abelian, then so is $G(H;e)$.
\end{Definition}
\begin{Definition}
Let $H$ be a heap. For all $e,e'\in H$, the \textit{translation isomorphism} $\tau^{e}_{e'}\colon H\to H$ is defined as
\[
\tau^{e}_{e'}(a) : = [a,e',e].
\]
\end{Definition}
One easily checks that the inverse of $\tau^{e}_{e'}$ is given by $\tau^{e'}_{e}$. Furthermore, $\tau^{e'}_{e}$ is an isomorphism of groups $G(H;e)\to G(H;e')$. The set of all translation isomorphisms of $H$ is a group with respect to the composition. This group is isomorphic to any of the retracts of $H$.

A (non-empty) subset $K$ of a heap $(H,[---])$ is called a \textit{sub-heap} if it is closed under the heap operation, that is, for all $a,b,c\in K$, $[a,b,c]\in K$. A sub-heap defines an equivalence relation on $H\colon a\sim b$ if and only if, for all (equivalently, any) $x\in K$, $[a,b,x]\in K$. The set of equivalence classes is denoted by $H/K$. One easily shows that this is the same as the heap associated to the quotient of retracts, that is, $H/K = H(G(H;e)/G(K;e))$, for any $e\in K$. If~$H$ is abelian, then $H/K$ is an abelian heap with the inherited structure $[\bar{a},\bar{b},\bar{c}] = \overline{[a,b,c]}$, where $\bar{a}$ etc.\ denotes the class of $a$ in $H/K$.

\begin{Definition}\label{truss}
A \textit{truss} is an algebraic structure $(T,[---],\cdot)$, consisting of a set $T$, a ternary operation $[---]$ such that $(T,[---])$ forms an abelian heap, and an associative binary operation $\cdot$ which distributes over $[---]$, that is, for all $a,b,c,d \in T$,
\[
 a[b,c,d] = [ab,ac,ad] \qquad \text{and} \qquad
 [b,c,d]a=[ba,ca,da].
\]

If the multiplication $\cdot$ admits identity, then the truss is said to be \textit{unital}. Finally, if the operation $\cdot$ is commutative, we then refer to the truss as being a \textit{commutative truss}.
\end{Definition}
\begin{Remark} We may use the definition of a retract (see Definition~\ref{ret}) as an alternative way of looking at the heap operation and the truss distributive laws in the retract. In $G(T;e)$, we have
\begin{align*}
& a(b+c) = a[b,e,c] = [ab,ae,ac] = ab-ae+ac, \qquad (a+b)c = ac-ec+bc,\\
& a(b-c) = a[b,c,e] = [ab,ac,ae] = ab-ac+ae, \qquad (a-b)c = ac-bc+ec, \\
& a(-b) = a(e-b)= ae-ab+ae, \qquad (-a)b = eb-ab+eb.
 \end{align*}
There are often times where thinking of the heap operation as a combination of the binary operations $+$ and $-$ will prove useful, and these distributive laws will be used throughout this paper.
\end{Remark}

\begin{Remark}\label{rem.trusses} The world of trusses is substantially richer than that of rings. On any abelian group $A$ understood as a heap $H(A)$ there are at least four non-isomorphic truss multiplications, only one of which gives rise to a ring in all circumstances. These are
\begin{gather}\label{truss.ab}
ab = 0, \qquad ab=a, \qquad ab=b, \qquad ab=a+b= [a,0,b]
\end{gather}
for all $a,b\in A$. There are additional truss structures on specific groups. For example, commutative truss multiplications $\cdot$ on $\Z$ are given in terms of the usual arithmetic operations on $\Z$, for all $m,n\in \Z$, by
\begin{gather}\label{truss.z}
 m\cdot n =amn +b(m+n) +c,
\end{gather}
where $a,b,c\in \Z$ are such that
\begin{gather}\label{truss.z.constraint}
 ac=b(b-1);
\end{gather}
see \cite[Theorem~3.51]{Brz:par}. We denote these trusses by $T(\Z;a,b,c)$. These split into isomorphism classes:
\begin{itemize}\itemsep=0pt
 \item[(1)] for all $a\in \N$,
\begin{align*}
m\cdot n = amn,\qquad
m\cdot n = amn + m +n,
\end{align*}
\item[(2)] for all $a\in \Z_+$, $b\in \{2,3,\dots,a-1\}$ and $c\in \Z_+$ such that $ac = b(b-1)$,
\begin{align*}
&m\cdot n = amn +b(m+n) +c,
\end{align*}
\end{itemize}
see \cite[Corollary~3.53]{Brz:par}. As explained in \cite[Example~7.4]{AndBrzRyb:ext}, there are 23 non-isomorphic truss structures on the group $\Z_p\oplus \Z_p$ ($p$ a prime number) as opposed to 8 ring structures.
\end{Remark}

\begin{Remark}\label{rem.ideal}
 As explained in \cite{AndBrzRyb:ext}, there is a close relation between trusses and ring extensions. More precisely, let $R$ be an associative ring, let $I$ be an ideal in $R$ and let $q\in R$ be an idempotent element. Then
$ T(I;q) := q +I$
 is a truss with the heap operation $[a,b,c]=a-b+c$ and the same multiplication as in $R$. Any truss can be embedded in an associative ring in this way.
 \end{Remark}

\section{Hochschild cohomology of trusses}
The aim of this section is to make a proposal for the Hochschild cohomology of trusses.
\begin{Definition}\label{def.cob}
Let $T$ be a truss and let $C^0(T) = T$ and, for all positive integers $n$, let $C^n(T)$ be the set of all multi-heap functions $T^n\to T$ (i.e., heap morphisms in each argument). For all $n \in \N$, $C^n(T)$ are viewed as heaps with the operation defined pointwise, that is inherited from~$T$.

For all $n \in \N$ and $e\in T$, the heap homomorphism $\delta_e^n\colon C^n(T)\to C^{n+1}(T)$ is defined by
\begin{align*}
 \delta^n_e f(a_0,\dots,a_n) =
 \begin{cases}
 [e, a_0 e, a_0 f(a_1,\dots,a_n), f(a_0a_1, a_2, \dots, a_n), \dots,& \\
 \quad f(a_0,a_1, a_2, \dots, a_{n-1}a_n), f(a_0,\dots,a_{n-1})a_n, ea_n], \quad & n \ \text{even},\\
 [e, a_0 e, a_0 f(a_1,\dots,a_n), f(a_0a_1, a_2, \dots, a_n), \dots, & \\
 \quad f(a_0,a_1, a_2, \dots, a_{n-1}a_n), f(a_0,\dots,a_{n-1})a_n, ea_n, e], \quad &n \ \text{odd} \end{cases}
\end{align*}
for all $f\in C^n(T)$ and $a_0, \dots ,a_n \in T$. We call this the \textit{e-relative Hochschild $n$-coboundary operator} on $T$.
\end{Definition}
We note in passing that the maps $\delta_e^n$ are heap homomorphisms by the truss distributive law and the fact that $T$ is an abelian heap. The usage of the term {\em coboundary} is justified by the following theorem.

\begin{Theorem} \label{deltathme}
For all $n\in \N$ and $e\in T$,
$\delta^{n+1}_e \circ \delta^{n}_e = e$,
where in each case the constant function with value $e$ is denoted by $e$ as well.

Furthermore, for all $n$, $\delta_e^n(e) = e$, and hence each $\delta^n_e$ is a~homomorphism of abelian groups $\delta_e^n\colon G(C^n(T);e) \to G\big(C^{n+1}(T);e\big)$.
\end{Theorem}

\begin{proof}
Here we shall represent the heap operation $[---]$ as a linear combination of~$+$ and~$-$ as well as using the distributive laws from the retract $G(T;e)$ where $(T,[---],\cdot)$ is a truss.
For the case, where $n$ is even and $n+1$ is odd,
\begin{align*}
 &\delta^n_e f(a_0,\dots,a_n)= - a_0 e + a_0 f(a_1,\dots,a_n) + \sum^{n}_{j=1}(-1)^j f (a_0,\dots,a_{j-1}a_j,\dots,a_n) \\
 & \hphantom{\delta^n_e f(a_0,\dots,a_n)=}{} -f(a_0,\dots,a_{n-1})a_n+ea_n,\\
 &\delta^{n+1}_e f(a_0,\dots,a_{n+1})= - a_0 e + a_0 f(a_1,\dots,a_{n+1}) + \sum^{n+1}_{i=1}(-1)^i f (a_0,\dots,a_{i-1}a_i,\dots,a_{n+1}) \\
 &\hphantom{\delta^{n+1}_e f(a_0,\dots,a_{n+1})=}{} +f(a_0,\dots,a_{n})a_{n+1}-ea_{n+1}.
\end{align*}
We compose these functions
\begin{gather*}
 \delta^{n+1}_e\delta^{n}_e f(a_0,\dots,a_{n+1})\\
\qquad{} =- a_0 e + a_0 \delta^n_e f (a_1,\dots,a_{n+1})+ \sum_{i=1}^{n+1} (-1)^i \delta^n_e f (a_0,\dots,a_{i-1}a_i,\dots,a_{n+1}) \\
\qquad\quad{} + \delta^n_e f(a_0,\dots,a_n)a_{n+1} - ea_{n+1}.
\end{gather*}
Cancelling terms with alternating signs, we find
\begin{gather*}
   \sum_{i=1}^{n+1} (-1)^i \delta^n_e f (a_0,\dots,a_{i-1}a_i,\dots,a_{n+1}) \\
   \qquad{} = a_0 a_1 e - a_0 a_1 f(a_2,\dots,a_{n+1})
 - \sum_{i=2}^{n+1}(-1)^{i+1}a_0 f(a_1,\dots,a_{i-1}a_i,\dots,a_{n+1}) \\
\qquad\quad{}  - \sum_{j=1}^n (-1)^j f(a_0,\dots,a_{j-1}a_j,\dots,a_n)a_{n+1}  + f (a_0,\dots,a_{n-1})a_n a_{n+1} - ea_n a_{n+1}.
\end{gather*}
This can then be simplified further by applying the definition of $\delta^n_e$ and the truss distributive laws
\begin{gather*}
 \sum_{i=1}^{n+1} (-1)^i \delta^n_e f (a_0,\dots,a_{i-1}a_i,\dots,a_{n+1})\\
 \qquad =  a_0e - a_0\delta^n_e f (a_1,\dots,a_{n+1}) - \delta^n_e f (a_0,\dots,a_{n})a_{n+1} + ea_{n+1}.
\end{gather*}
Making the substitution in the formula for the composition, we thus find
\begin{gather*}
\delta^{n+1}_e\delta^{n}_e f(a_0,\dots,a_{n+1})\\
\qquad {}= - a_0 e
 +a_0\delta^n_e f (a_1,\dots,a_{n+1}) +a_0e - a_0\delta^n_e f (a_1,\dots,a_{n+1})
  - \delta^n_e f (a_0,\dots,a_{n})a_{n+1} \\
  \qquad\quad{} + ea_{n+1} + \delta^n_e f (a_0,\dots,a_{n})a_{n+1}  -ea_{n+1}=e.
\end{gather*}

 Now we may look at the case where $n$ is odd and $n+1$ is even.
In this case the composition comes out as
\begin{gather*}
 \delta^{n+1}_e\delta^{n}_e f(a_0,\dots,a_{n+1})\\
 \qquad{} = - a_0 e + a_0 \delta^n_e f (a_1,\dots,a_{n+1})  + \sum_{i=1}^{n+1} (-1)^i \delta^n_e f (a_0,\dots,a_{i-1}a_i,\dots,a_{n+1}) \\
\qquad\quad{} - \delta^n_e f(a_0,\dots,a_n)a_{n+1} + ea_{n+1}.
 \end{gather*}
Once more, cancelling the alternating terms using the definition of functions $\delta^n_e$ as well as the distributive laws for trusses, we find
\begin{gather*}
  \sum_{i=1}^{n+1} (-1)^i \delta^n_e f (a_0,\dots,a_{i-1}a_i,\dots,a_{n+1}) \\
  \qquad{}= a_0 a_1 e - a_0 a_1 f(a_2,\dots,a_{n+1})
 - \sum_{i=2}^{n+1}(-1)^{i+1}a_0 f(a_1,\dots,a_{i-1}a_i,\dots,a_{n+1}) \\
\qquad\quad{} + \sum_{j=1}^n (-1)^j f(a_0,\dots,a_{j-1}a_j,\dots,a_n)a_{n+1} \\
 \qquad\quad{} + e a_{n+1} -a_0 e
 + f (a_0,\dots,a_{n-1})a_n a_{n+1} -  ea_n a_{n+1} \\
\qquad{} = a_0e - a_0\delta^n_e f (a_1,\dots,a_{n+1})+ \delta^n_e f (a_0,\dots,a_{n})a_{n+1} - ea_{n+1}.
\end{gather*}
Using the above calculation, we then expand
\begin{gather*}
\delta^{n+1}_e \delta^{n}_e f(a_0,\dots,a_{n+1})\\
\qquad{}  =
- a_0 e + a_0 \delta^n_e f (a_1,\dots,a_{n+1}) + a_0e - a_0\delta^n_e f (a_1,\dots,a_{n+1})+ \delta^n_e f (a_0,\dots,a_{n})a_{n+1}\\
\qquad\quad{}- ea_{n+1}  - \delta^n_e f(a_0,\dots,a_n)a_{n+1} + ea_{n+1} = e
\end{gather*}
as required.

Since $e$ is the neutral element in $G(T;e)$, we immediately find that, for all $a_0,\dots, a_n$ and the constant function $e\colon T^n\to T$, $(a_0\dots a_{n-1})\mapsto e$,
\[
\delta_e^n(e)(a_0,\dots, a_n) = -a_0e +a_0e +(-1)^{n+1}ea_n + (-1)^n ea_n = e.
\]
Now, the observation that any homomorphism of heaps $f\colon H\to K$ is a homomorphism of groups $f\colon G(H;e)\to G(K;f(e))$, for all $e\in H$ confirms the final assertion.
\end{proof}

\begin{Definition}\label{def.cob.coc}
 In the setup of Definition~\ref{def.cob}, we define the heap of {\em $e$-relative $n$-cocycles}
 \[
 Z^n_e(T) := \{f\in C^{n}(T)\; |\; \delta_e^n(f) =e\},
 \]
 and the heap of {\em $e$-relative $n$-coboundaries}
 \[
 B^n_e(T):= \im \delta_e^{n-1}.
 \]
\end{Definition}

Since Theorem~\ref{deltathme} implies that $B^n_e(T)$ is a sub-heap of $Z^n_e(T)$, we can formulate the definition of the main object of study in this section.
\begin{Definition}\label{quotient.def}
For all $n\in \N$ and $e\in T$, the quotient heap
\[ H^n_e(T) = Z^n_e(T)/B^n_e(T)
\]
is called the {\em $n$-th $e$-relative Hochschild cohomology heap} of $T$.
\end{Definition}

\begin{Remark}\label{rem.coh.gr}
 The last assertion of Theorem~\ref{deltathme} ensures the existence of the cochain complex of abelian groups $\big(\delta_e^n\colon G(C^n(T);e) \to G\big(C^{n+1}(T);e\big)\big)$. This allows one for an abelian group interpretation of $e$-retracts of $e$-relative Hochschild cohomology heaps, namely, $G(H^n_e(T);\bar{e})$, where $\bar{e}$ is the class of the constant heap map $e$, is the $n$-th cohomology group of the above complex of abelian groups $G(C^n(T);e)$.
\end{Remark}
We next show that the relative Hochschild cohomology heaps for different $e$ can be identified up to isomorphism. We start with the following lemma.

\begin{Lemma} \label{complem}
For all heap homomorphisms $f\in C^n(T)$ and $e,e'\in T$,
\[
\delta^{n}_{e'} ( f )= \tau_{e}^{e'} \circ \delta^{n}_{e} ( \tau^{e}_{e'} \circ f). \]
\end{Lemma}
\begin{proof}
First, we look at the case where $n$ is even,
\[\delta^{n}_{e'} f(a_0,\dots,a_n) = \tau^{e'}_{e} (\delta^{n}_{e}(\tau^{e}_{e'} \circ f)(a_0,\dots,a_n) ). \]
Let us first calculate
\begin{align*}
 &\delta^n_{e} (\tau^e_{e'} \circ f)(a_0 ,\dots, a_n) = [e,a_0 e, a_0 f(a_1,\dots,a_n), a_0 e', a_0 e,f(a_0 a_1,\dots,a_n),e',e,\dots, \\
 & \hphantom{\delta^n_{e} (\tau^e_{e'} \circ f)(a_0 ,\dots, a_n) =[}{} f(a_0,\dots,a_{n-1}a_n),e',e,f(a_0,\dots,a_{n-1})a_n,e' a_n,ea_n,ea_n], \\
 & \hphantom{\delta^n_{e} (\tau^e_{e'} \circ f)(a_0 ,\dots, a_n) }{} =[e,a_0 e', a_0 f(a_1,\dots,a_n),f(a_0 a_1,\dots,a_n),\dots, \\
 & \hphantom{\delta^n_{e} (\tau^e_{e'} \circ f)(a_0 ,\dots, a_n) =[}{}  f(a_0,\dots,a_{n-1} a_n), f(a_0,\dots,a_{n-1})a_n,e' a_n], \\
 & \hphantom{\delta^n_{e} (\tau^e_{e'} \circ f)(a_0 ,\dots, a_n)}{} = [e,e', \delta^n_{e'} f(a_0,\dots,a_n)].
\end{align*}
Now, for the case where we have $n$ odd, by similar calculation,
\begin{gather*}
 \delta^n_{e} (\tau^e_{e'} \circ f)(a_0 ,\dots, a_n) = [e,a_0 e, a_0 f(a_1,\dots,a_n), a_0 e', a_0 e,f(a_0 a_1,\dots,a_n),e',e, \\
\hphantom{\delta^n_{e} (\tau^e_{e'} \circ f)(a_0 ,\dots, a_n) = [}{}
 f(a_0, a_1 a_2,\dots,a_n),e',e, \dots,f(a_0,\dots,a_{n-1}a_n),e',e, \\
\hphantom{\delta^n_{e} (\tau^e_{e'} \circ f)(a_0 ,\dots, a_n) = [}{} f(a_0,\dots,a_{n-1})a_n,e' a_n,ea_n,ea_n,e] \\
\hphantom{\delta^n_{e} (\tau^e_{e'} \circ f)(a_0 ,\dots, a_n) }{} = [e,a_0 e,a_0 f(a_1,\dots,a_n), f(a_0 a_1,\dots,a_n),e',e, \\
\hphantom{\delta^n_{e} (\tau^e_{e'} \circ f)(a_0 ,\dots, a_n) = [}{} f(a_0,a_1 a_2,\dots,a_n),\dots,f(a_0,\dots,a_{n-1}a_n), \\
\hphantom{\delta^n_{e} (\tau^e_{e'} \circ f)(a_0 ,\dots, a_n) = [}{} f(a_0,\dots,a_{n-1})a_n,e' a_n,e]\\
\hphantom{\delta^n_{e} (\tau^e_{e'} \circ f)(a_0 ,\dots, a_n)}{} =[e,e', \delta^n_{e'} f(a_0,\dots,a_n)].
\end{gather*}
Thus, for both cases,
\begin{align*}
\tau^{e'}_e (\delta^n_e (\tau^e_{e'} \circ f)(a_0,\dots,a_n)) = [e,e',\delta^n_{e'} f(a_0,\dots,a_n),e,e'] = \delta^n_{e'} f(a_0,\dots,a_n).\tag*{\qed}
\end{align*} \renewcommand{\qed}{}
\end{proof}

Thanks to the commutativity of the heap operations and the Mal'cev identities, for all ${e,e'\in T}$ and $n\in \N$, we can consider heap isomorphisms
\[
\vec{\tau}_e^{\,e'}\colon\ C^n(T)\to C^n(T), \qquad f\mapsto \tau_e^{e'}\circ f.
\]
In terms of these isomorphisms, the statement of Lemma~\ref{complem} can be rephrased as
\begin{gather}\label{eq.tau}
 \delta_{e'}^n \big(\vec{\tau}_e^{\,e'}(f)\big) = \vec{\tau}_e^{\,e'}(\delta_e^n(f))
\end{gather}
for all $f\in C^n(T)$.
\begin{Lemma}\label{lem.b.z}
For all $e,e' \in A$, the maps $\vec{\tau}_e^{\,e'}$ restrict to isomorphisms $B^n_e(T) \to B^n_{e'}(T)$ and~$Z^n_e(T) \to Z^n_{e'}(T)$.
\end{Lemma}
\begin{proof}
This follows (almost) immediately from Lemma~\ref{complem} or its equivalent formulation in~\eqref{eq.tau}. Specifically, if $g = \delta_e^n(f)\in B_e^n(T)$, then
$\vec{\tau}_e^{\,e'}(g) = \delta_{e'}^n\big(\vec{\tau}_e^{\,e'}(f)\big) \in B_{e'}^n(T)$. If $\delta_e^n(f) = e$, then
\[e' = [e,e,e'] = \vec{\tau}_e^{\,e'}\big(\delta_e^n(f)\big) = \delta^{n}_{e'}\big(\vec{\tau}_e^{\,e'}(f)\big),
\]
i.e., $\vec{\tau}_e^{\,e'}(f)\in Z_{e'}^n(T)$.
\end{proof}

Put together, Lemmas~\ref{complem} and~\ref{lem.b.z} yield the identification of cohomology heaps sought for.
\begin{Theorem}\label{thm.Hoch}
 Let $T$ be a truss. Then, for all $n\in \N$ and $e,e'\in T$,
 \[
 H^n_e(T) \cong H^n_{e'}(T).
 \]
\end{Theorem}
\begin{proof}
 By Lemmas~\ref{complem} and~\ref{lem.b.z}, the isomorphisms $\vec{\tau}_e^{\,e'}$ descent to the isomorphisms $H^n_e(T)\to H^n_{e'}(T)$; for $f,g\in Z_e^n(T)$ represent the same class in $H_e^n(T)$ if and only if
 \[\big[f,g,\delta^{n-1}_e(h)\big] \in B_e^n(T)\]
 for some (equivalently all) $h\in C^{n-1}(T)$. Hence,
 \begin{align*}
 \big[\vec{\tau}_e^{\,e'}(f),\vec{\tau}_e^{\,e'}(g),\delta^{n-1}_{e'}\big(\vec{\tau}_e^{\,e'}(h)\big)\big] &=\big[\vec{\tau}_e^{\,e'}(f),\vec{\tau}_e^{\,e'}(g),\vec{\tau}_e^{\,e'}\big(\delta_e^{n-1}(h)\big)\big]\\
 &= \vec{\tau}_e^{\,e'}\left(\left[f,g,\delta_e^{n-1}(h)\right]\right)\in B_{e'}^n(T).\tag*{\qed}
 \end{align*}\renewcommand{\qed}{}
\end{proof}
In view of Theorem~\ref{thm.Hoch} rather than talking about $e$-relative Hochschild cohomology heaps, we might talk just as well about simply Hochschild cohomology heaps and drop the subscript $e$ from the notation.

\begin{Remark}\label{rem.iso.gr}
 One easily checks that $\vec{\tau}_e^{\,e'}(e) =e'$, and hence the isomorphism described in Theorem~\ref{thm.Hoch} is an isomorphism of abelian groups $G(H^n_e(T);\bar{e})\simeq G(H^n_{e'}(T); \overline{e'})$.
\end{Remark}

In case of the Hochschild cohomology of algebras, one-cocycles correspond to derivations. A~similar statement can be made in the case of the cohomology of trusses, although this correspondence is not quite as direct as in the ring case.

\begin{Definition}[\cite{Brz:Lie}]
 Let $T$ be a truss. A heap homomorphism $D\colon T\to T$ is called a~\textit{derivation} if, for all $a,b\in T$,
 \[
 D(ab) = [D(a)b, ab, aD(b)].
 \]
 Derivations on $T$ form a heap which is denoted by $\operatorname{Der}(T)$.
\end{Definition}

\begin{Proposition}
 For all $e\in T$, $\operatorname{Der}(T)\cong Z_e^1(T)$ as heaps.
\end{Proposition}
\begin{proof}
 The isomorphism and its inverse are given by
 \begin{alignat*}{3}
 &\Theta\colon\ \operatorname{Der}(T)\to Z^1_e(T),\qquad && \Theta(D)\colon\ a\mapsto [D(a),a,e],&
 \\
 & \Theta^{-1}\colon\ Z^1_e(T) \to \operatorname{Der}(T), \qquad&& \Theta^{-1}(f)\colon\ a\mapsto [f(a),e,a].&
 \end{alignat*}
It is clear that the defined maps are inverses of each other. Thus, it remains to be checked if their domains and codomains are as stated.

If $D$ is a derivation, then, for all $a,b\in T$,
\begin{align*}
 \delta_e^1(\Theta(D))(a,b) &= [e,ae,a\Theta(D)(b), \Theta(D)(ab), \Theta(D)(a)b, eb, e]\\
 &= [e,ae,aD(b),ab,ae, D(ab), ab,e,D(a)b, ab, eb,eb,e]\\
 &=[e,D(ab), aD(b),ab, D(a)b] = [e,D(ab), D(ab)] =e,
\end{align*}
where the second equality uses the truss distributive law, the third one arises from the cancellation and reshuffling rules described in Remark~\ref{rem.rules}, and the penultimate equality is the definition of the derivation. Thus, $\Theta(D) \in Z_e^1(T)$ as required.

In the converse direction, if $f\in Z_e^1(T)$, then, for all $a,b\in T$,
\[
[e, ae, af(b), f(ab), f(a)b, eb, e]= e,
\]
and so by Remark~\ref{rem.equal}, equivalently
\[
f(ab) = [f(a)b, eb, af(b),ae, e].
\]
Therefore,
\begin{align*}
 \big[\Theta^{-1}(f)(a)b,ab,a\Theta^{-1}(f)b\big] = [f(a)b,eb,ab,ab,af(b),ae,ab]= [f(ab),e,ab] = \Theta^{-1}(f)(ab),
\end{align*}
so that $\Theta^{-1}(f)$ is a derivation on $T$ as required.
\end{proof}

As an illustration of Hochschild cohomology of trusses, we compute the cohomology heaps of the second of the trusses \eqref{truss.ab} in Remark~\ref{rem.trusses}.

\begin{Example}\label{ex.HH.ab}
 For an abelian group $A$, let us denote by $\mathcal{L}(A)$ the truss with the product given by the left projection, that is $ab =a$, then
 \begin{align*}
 H^n(\mathcal{L}(A)) \cong \begin{cases}
 \operatorname{End}(A) & \mbox{for}\ n=1, \cr
 \ast & \mbox{otherwise}.
 \end{cases}
 \end{align*}
\end{Example}
\begin{proof}
 We perform all computations relative to the neutral element $0$ in $A$. In view of the product in $\mathcal{L}(A)$, the formula for the coboundary operator (relative to $0$ and written in the abelian group form) comes out as
 \begin{gather}\label{dnf}
 \delta^n f(a_0,\dots,a_n) = \sum^{n-1}_{i=1}(-1)^i f \big(a_0,\dots ,\widehat{a_i},\dots,a_n\big),
\end{gather}
where $\widehat{a_i}$ indicates the absence of $a_i$.

First, note that, for all $a,b\in A$,
 \begin{gather}\label{0.cob}
 \delta^0(a)(b) = -a,
 \end{gather}
 and hence $H^0(\mathcal{L}(A))= Z^0(\mathcal{L}(A)) =\{0\}$.

 For $n\geq 2$ and $f\in C^n(\mathcal{L}(A))$, $\delta^n (f) = 0$ if and only if, for all $a_0,\dots, a_n\in A$,
 \begin{align}
 f(a_0,a_2,\dots , a_n)={}& f(a_0,a_1,a_3,\dots, a_n) - f(a_0,a_1,a_2,a_4,\dots, a_n)+\cdots \nonumber\\
 &+ (-1)^{n+1}f(a_0,a_1,a_2,\dots, a_{n-2}, a_n).\label{n.form}
 \end{align}
 The left-hand side of equation \eqref{n.form} is independent of $a_1$. Thus, setting $a_1=0$ and relabelling the indices, we find that $f$ is an $n$-cocycle if and only if, for all $a_0,\dots, a_{n-1}$,
 \begin{align}
 f(a_0,a_1,\dots , a_{n-1}) ={}& f(a_0,0, a_2,\dots, a_{n-1}) - f(a_0,0,a_1,a_3,\dots, a_{n-1})+\cdots\nonumber\\
 & + (-1)^{n+1}f(a_0,0,a_1,\dots, a_{n-3}, a_{n-1}).\label{n.constraint}
 \end{align}
 Set
 \[
 g\colon\ A^{n-1} \to A, \qquad (a_0,a_1,\dots , a_{n-2})\mapsto -f(a_0,0,a_1,\dots , a_{n-2}).
 \]
 Since $f\in C^n(\mathcal{L}(A))$, the map $g$ is a heap homomorphism in all arguments, and hence $g\in C^{n-1}(\mathcal{L}(A))$. The formula \eqref{n.constraint} immediately yields $\delta^{n-1}(g) =f$, and hence every $n$-cocycle is also an $n$-coboundary. Therefore, all Hochschild heaps are trivial whenever $n\geq 2$.

 Finally, elements of $C^1(\mathcal{L}(A))$ are heap endomorphisms of $f$, i.e., any functions $f\colon A\to A$ such that $f(a-b+c) = f(a)-f(b) +f(c)$. Given such an $f$, the map $g(a) = f(a) -f(0)$ is additive. Conversely, given a group endomorphism $g$ of $A$ and $c\in A$, the map $f(a) = g(a) +c$ is a heap endomorphism. The formula \eqref{dnf} implies that every heap endomorphism $f\colon A\to A$ is a~one-cocycle, while \eqref{0.cob} yields that two one-cocycles belong to the same cohomology class if and only if they differ by a~constant (i.e., they correspond to the same abelian group endomorphism). This establishes the isomorphism of $H^1(\mathcal{L}(A))$ with the heap of additive endomorphisms of~$A$.
\end{proof}

\section{Nijenhuis products and operators on trusses}\label{section4}

In this section, we transfer the notions of Nijenhuis products and operators introduced in~\cite{4} from rings to trusses, and determine the sufficient and necessary conditions for the associativity of the Nijenhuis product. We also study compatibility of Nijenhuis operators and give general examples as well as classify all Nijenhuis operators on commutative trusses built on the group of integers (see Remark~\ref{rem.trusses}).
\begin{Definition}\label{def.Nij.tor}
Let $T$ be a truss. For all heap homomorphisms $N\colon T \rightarrow T$, the binary operation~$\circ_{N}$ on~$T$, defined by
\[a \circ_{N} b = [N(a)b, N(ab), aN(b)],\]
is called the \textit{Nijenhuis product}.

Furthermore, $N$ is called a
\textit{Nijenhuis operator} if, for all $a,b \in T$,
\[N(a \circ_{N} b) = N(a)N(b).\]
For all $e \in T$, the \textit{$e$-Nijenhuis torsion} of $N$ is defined as
\[T^{e}_{N}(a,b) = [N(a \circ_{N} b),N(a)N(b),e].\]
\end{Definition}

Note that in view of Remark~\ref{rem.equal}, $N$ is a Nijenhuis operator if and only if its $e$-torsion is a~constant function equal to $e$, that is, for all $e,a,b\in T$, $T^{e}_{N}(a,b)=e$. Such an $e$-torsion is said to be \textit{trivial}.

\begin{Remark}
 The authors of \cite{4} use the term Nijenhuis tensor rather than Nijenhuis operator. However, as the former commonly is used to describe the obstruction of an almost complex structure to originate from a complex structure and is closer to the Nijenhuis torsion (see \cite[footnote~1, p.~627]{Kos:Nij}), we prefer the latter. Besides the term Nijenhuis operator is now widely used to describe a way of deforming of a given algebraic structure (typically a Lie bracket, but associative products too) into a structure of the same kind, which also extends to trusses as argued in the present text.
\end{Remark}

\begin{Example}\label{ex.id.proj}
 Let $T$ be a truss.
 \begin{itemize}\itemsep=0pt
 \item[(1)] The identity map $\id\colon T\to T$ is a Nijenhuis operator.
 \item[(2)] Let $P\colon T\to T$ be a multiplicative idempotent homomorphism of heaps. Then $P$ is a~Nijenhuis operator on $T$. In particular, for any idempotent element $q\in T$, the constant map $a\mapsto q$ is a~Nijenhuis operator.
 \end{itemize}
\end{Example}
\begin{proof}
 In the first case, the Nijenhuis product is the same as the original multiplication in $T$ and hence clearly the identity map is a Nijenhuis operator. In the second example, since, for all~$a,b,c\in T$,
 \[
 P([a,b,c]) = [P(a), P(b), P(c)], \qquad P(ab) = P(a)P(b), \qquad P(P(a)) =P(a),
 \]
 we can easily compute
 \begin{align*}
 P(a\circ_Pb) &= P([P(a)b, P(ab),aP(b)]) = [P(P(a)b), P(a)P(b),P(aP(b)]\\
 &= [P(a)P(b), P(a)P(b),P(a)P(b)] = P(a)P(b)
 \end{align*}
 as required.

 Any constant map is a heap homomorphism of heaps, and if the image is an idempotent element of $T$, then such a map is a multiplicative idempotent homomorphism of heaps.
\end{proof}

A direct connection between Nijenhuis operators on rings and trusses is given in the following proposition.
\begin{Proposition}\label{prop.Nij.ideal}
 Let $T(q;I)$ be a truss associated to an ideal $I$ and idempotent $q$ in a ring $R$; see Remark~{\rm \ref{rem.ideal}}. Let $\bar{N}$ be a Nijenhuis operator on $I$, such that, for all $x\in I$,
 \begin{gather}\label{n.e}
 \bar{N}(xq) = \bar{N}(x)q, \qquad \bar{N}(qx) =q\bar{N}(x).
 \end{gather}
 Then
 \[
 N\colon\ T(q;I)\to T(q;I), \qquad q+x\mapsto q+ \bar{N}(x)
 \]
 is a Nijenhuis operator on $T(q;I)$.
\end{Proposition}
\begin{proof}
 It is immediate from the definition of $N$ that it is a heap homomorphism. The Nijenhuis product comes out as, for all $x,y\in I$,
 \begin{align*}
 (q+x)\circ_N(q+y) &= q +\bar{N}(x)q + qy +\bar{N}(x)y - q - \bar{N}(xq) - \bar{N}(qy) - \bar{N}(xy)\\
 & \quad+ q + x \bar{N}(y) + xq +q \bar{N}(y)\\
 &= q + x\circ_{\bar{N}}y +xq +qy
 \end{align*}
 by the fact that $q$ is an idempotent and by \eqref{n.e}. Since $\bar{N}$ is the Nijenhuis operator on the ring~$I$, we conclude
 \[
 N\left((q+x)\circ_N(q+y)\right) = q + \bar{N}(x)\bar{N}(y) + \bar{N}(x)q + q\bar{N}(y) = N(q+x)N(q+y),
 \]
that is, $N$ is a Nijenhuis operator on the truss $T(I;q)$.
\end{proof}

\begin{Example}\label{ex.proj}
An explicit example of a Nijenhuis operator of the type described in Proposition~\ref{prop.Nij.ideal} can be constructed as follows. Let $R$ be the subring of the ring $M_{n\times n}(\F)$ of $(n+1)\times (n+1)$-matrices over a field $\F$ consisting of matrices of the following block form
\[
A(\mathbf{a}, \mathbf{b}, \alpha) =
\begin{pmatrix}
 \mathbf{a} & \mathbf{b}\\
 0 & \alpha
\end{pmatrix},
\]
where $\mathbf{a}\in M_{n\times n}(\F)$, $\mathbf{b}\in M_{n\times 1}(\F)$, and $\alpha \in \F$. Set
\[
I = \{A(\mathbf{a}, \mathbf{b}, 0)\;|\; \mathbf{a}\in M_{n\times n}(\F), \mathbf{b}\in M_{n\times 1}(\F)\}, \qquad q= A(0,0,1),
\]
so that
\[
T(I;q) = \{A(\mathbf{a}, \mathbf{b}, 1)\;|\; \mathbf{a}\in M_{n\times n}(\F), \mathbf{b}\in M_{n\times 1}(\F)\}.
\]
Let $P_+$ denote the projection in $M_{n\times n}(\F)$ onto the subring of upper triangular matrices. As argued in \cite[Example~1]{4}, the map
\[
\bar{N}\colon\ I\to I, \qquad A(\mathbf{a}, \mathbf{b}, 0) \mapsto A(P_+(\mathbf{a}), \mathbf{b}, 0)
\]
is a Nijenhuis operator on $I$. Since $\bar{N}$ affects only the upper $n\times n$ block of $A(\mathbf{a}, \mathbf{b}, 0)$ and in view of the general multiplication rules in~$R$,
\[
A(\mathbf{a}, \mathbf{b}, \alpha)A(\mathbf{a'}, \mathbf{b'}, \alpha') = A(\mathbf{a}\mathbf{a'}, \mathbf{a}\mathbf{b'} +\alpha' \mathbf{b}, \alpha \alpha'),
\]
the condition \eqref{n.e} is satisfied. Therefore,
\[
{N}\colon\ T(I;q) \to T(I;q), \qquad A(\mathbf{a}, \mathbf{b}, 1) \mapsto A(P_+(\mathbf{a}), \mathbf{b}, 1)
\]
is a Nijenhuis operator on $T(I;q)$.
\end{Example}

\begin{Theorem} \label{2nd}
Let $T$ be a truss and $N\colon\ T\to T$ a heap homomorphism.
The Nijenhuis product~$\circ_{N}$ is associative if and only the $e$-Nijenhuis torsion of $N$ is an $e$-relative Hochschild $2$-cocycle for all $($equivalently for any$)$ $e\in T$. If this is the case, then $T$ is a truss with the Nijenhuis product~$\circ_{N}$.
\end{Theorem}
\begin{proof}
First, note that whether associative of not, the Nijenhuis product distributes over the heap operation. Indeed, for all $e,x,y,z\in T$,
\begin{align*}
w \circ_{N}[x,y,z]
&=\left[N(w)[x,y,z],N(w[x,y,z]),wN([x,y,z])\right] \\
&= \left[N(w)x,N(w)y ,N(w)z, N(wx),N(wy),N(wz), wN(x) ,wN(y) , wN(z)\right]\\
 &= \left[N(w)x,N(wx),wN(x), N(w)y,N(wy),wN(y), N(w)z,N(wz),wN(z)\right] \\
&=
[w \circ_{N}x,w \circ_{N}y,w \circ_{N}z],
\end{align*}
since $T$ is a truss and $N$ is a heap homomorphism. The rearrangement of terms leading to the third equality follows by the fact that $T$ is an abelian heap and hence the terms in odd (resp.\ even) positions in the square bracket can be reshuffled freely; see Remark~\ref{rem.rules}. This proves the left distributive law. The right distributive law is proven by similar calculations.

Using the heap homomorphism property of $N$, the associativity of the product in $T$ as well as its distributivity over the heap operation, we find, for all $a,b,c\in T$,
\begin{align*}
(a\circ_{N}b)\circ_{N}c ={}& [N(N(a)b)c, N(N(ab))c, N(aN(b))c, N(N(a)bc), \\
& N(N(ab)c), N(aN(b)c), N(a)bN(c),N(ab)N(c),aN(b)N(c)],
\end{align*}
and
\begin{align*}
a\circ_{N}(b\circ_{N}c) ={}&
[ N(a)N(b)c, N(a)N(bc) , N(a)bN(c),
 N(aN(b)c), \\
 & N(aN(bc)), N(abN(c)),
 aN(N(b)c),aN(N(bc)), aN(bN(c))].
\end{align*}
Therefore, cancelling repeated terms we obtain
\begin{align*}
[a \circ_{N} (b \circ_{N} c) &, (a \circ_{N}b) \circ_{N}c, e] = [N(a)N(b)c ,N(a)N(bc), N(aN(bc)), N(abN(c)), \\
& aN(N(b)c),aN(N(bc)),
 aN(bN(c)), N(N(a)b)c , N(N(ab))c, \\
&N(aN(b))c , N(N(a)bc) , N(N(ab)c),
 N(ab)N(c) , aN(b)N(c) ,e].
\end{align*}
Next, using the truss distributive laws, we compute
\begin{align*}
&aT^{e}_{N}(b,c) = [aN(N(b)c) , aN(N(bc)) ,aN(bN(c)) , aN(b)N(c) , ae],\\
&T^{e}_{N}(ab,c) = [N(N(ab)c) , N(N(abc)) ,N(abN(c)) , N(ab)N(c) , e],\\
&T^{e}_{N}(a,bc) = [N(N(a)bc) , N(N(abc)) , N(aN(bc)) , N(a)N(bc) , e],
\end{align*}
and
\begin{align*}
T^{e}_{N}(a,b)c
&= [N(N(a)b)c , N(N(ab))c , N(aN(b))c , N(a)N(b)c , ec].
\end{align*}
Thus,
\begin{align*}
\delta^2_e T^{e}_{N}(a,b,c) ={}& [aN(N(b)c), aN(N(bc)) , aN(bN(c)) , aN(b)N(c) , ae, \\
& N(N(ab)c) , N(N(abc)) , N(abN(c)) , N(ab)N(c) , e, \\
& N(N(a)bc) , N(N(abc)) , N(aN(bc)) , N(a)N(bc) , e, \\
& N(N(a)b)c , N(N(ab))c , N(aN(b))c , N(a)N(b)c , ec,
ec ,ae,e].
\end{align*}
We can then perform cancellations and rearrangements allowed by the definition of an abelian heap (see Remark~\ref{rem.rules}), yielding
\[
[a \circ_{N} (b \circ_{N} c), (a \circ_{N}b) \circ_{N}c, e] = \delta^2_e T^{e}_{N}(a,b,c).
\]
Therefore, by Remark~\ref{rem.equal},
\[
a \circ_{N} (b \circ_{N} c) = (a \circ_{N}b) \circ_{N}c,
\]
if and only if $\delta^2_e T^{e}_{N}(a,b,c)=e$ as required.
\end{proof}

\begin{Corollary}\label{cor.Nij}
If $N$ is a Nijenhuis operator on a truss $(T,[---],\cdot)$, it follows that $(T,[---],\circ_{N})$ is also a truss. We denote this truss by $T[N]$.
\end{Corollary}

\begin{proof}
By the definition of a Nijenhuis operator, its $e$-torsion is trivial for all $e$, and hence it is a 2-cocycle as needed.
\end{proof}

As an illustration of Theorem~\ref{2nd} we classify all Nijenhuis operators on $T(\Z;a,b,c)$ on $\Z$ described in Remark~\ref{rem.trusses}.

\begin{Proposition}\label{prop.Nij.z}
The following table lists all Nijenhuis operators on commutative trusses $T(\Z;a,b,c)$ on $\Z$ described in Remark~{\rm \ref{rem.trusses}}.
 \begin{center}\renewcommand{\arraystretch}{1.8}
 \begin{tabular}{| c | c c |}
\hline
 $T(\Z;a,b,c)$ & {\rm Nijenhuis operators} &
 \\
\hline $c\neq 0$
 & $N(m) = \left(\dfrac{b}{\gcd(b,c)} q +1\right)m + \dfrac{c}{\gcd(b,c)} q,$
 &
 \\[1ex]
& $N(m) = \left(\dfrac{b-1}{\gcd(b-1,c)} q +1\right)m + \dfrac{c}{\gcd(b-1,c)} q,$ & $q\in \Z$ \\
\hline $c=0$
 & $N(m) = qm,$
 & {}
 \\
{} & $N(m)= \left(\dfrac{a}{2b-1} q +1\right) m +q $,
 & $q\in \Z$ \\ \hline
 \end{tabular}
 \end{center}
 \end{Proposition}

\begin{proof}
A heap homomorphism $N\colon \Z\to \Z$ is necessarily of the form, for all $m\in \Z$,
 \begin{gather*}
 N(m) = pm +q
 \end{gather*}
for some $p,q\in \Z$. We need to determine what conditions $p$ and $q$ have to satisfy in order for $N$ to be a Nijenhuis operator. We take the most general commutative truss on $\Z$, $T(\Z;a,b,c)$, with multiplication \eqref{truss.z}, where $a$, $b$, $c$ satisfy the constraint \eqref{truss.z.constraint}, and compute, for all $m,n\in \Z$,
\begin{align*}
 &N(m)\cdot n = apmn +(aq+b)n +bpm +bq +c,\\
 &N(m\cdot n) = apmn +bpm +bpn +cp +q.
\end{align*}
These yield the Nijenhuis product
\begin{gather}\label{nij.prod.z}
m\circ_N n = apmn + (aq+b)(m+n) +2bq +2c -q -cp.
\end{gather}
In view of Theorem~\ref{2nd}, the product \eqref{nij.prod.z} if associative if and only if the $e$-Nijenhuis torsion is a cocycle for any fixed $e\in \Z$, in particular, for~0. The $0$-Nijenhuis torsion comes out as
\begin{equation}\label{0.torsion}
 T_N^0(m,n)= -cp^2+ ((2b-1)q +2c)p - aq^2 - (2b-1)q -c.
\end{equation}
On the other hand, as recalled in Remark~\ref{rem.trusses}, the product \eqref{nij.prod.z} is associative if and only if
\begin{equation}\label{z.constraint}
 2 abpq + 2acp - apq-acp^2 = (aq+b)(aq+b-1).
\end{equation}
Since $ac=b(b-1)$ and in view of \eqref{0.torsion}, this can be rewritten as
\begin{gather*}\label{constraint}
 aT_N^0(m,n)=0.
\end{gather*}
Thus, if $a\neq 0$, the associativity of the Nijenhuis product is equivalent to the vanishing of the Nijenhuis torsion, i.e., to $N$ being a Nijenhuis operator.

We consider the equation $T_N^0(m,n)=0$ as an equation with the unknown $p$. If $c\neq 0$ this is a quadratic equation with the discriminant $q^2$, and hence the solutions are
\[
p= \left(\frac{b}{c} q +1\right)\qquad \mbox{or}\qquad p= \left(\frac{b-1}{c} q +1\right).
\]
To ensure that the solutions are integer, the numbers $q$ must be multiples of the denominators of the fractions divided by the greatest common multiple of the numerator and the denominator. By rescaling accordingly, we obtain the following Nijenhuis operators:
\[
N(m) = \left(\frac{b}{\gcd(b,c)} q +1\right)m + \frac{c}{\gcd(b,c)} q, \qquad q\in \Z,
\]
and
\[
N(m) = \left(\frac{b-1}{\gcd(b-1,c)} q +1\right)m + \frac{c}{\gcd(b-1,c)} q, \qquad q\in \Z.
\]
If $c=0$, then $T_N^0(m,n)=0$ is equivalent to
\[
(2b-1)pq - aq^2 - (2b-1)q=0.
\]
Thus, $q=0$ or $p=\frac{a}{2b-1}\, q +1$, which is an integer, for all $q$, since the constraint \eqref{z.constraint} implies that $b=0$ or $b=1$ in this case.

Putting all these cases together, we obtain the table as stated.
\end{proof}

The following theorem is the truss version of \cite[Theorem~2]{4}.
\begin{Theorem}\label{thm.Nij.iter}
If $N$ is a Nijenhuis operator on a truss $T$, then, for all $j,k\in \N$,
\begin{itemize}\itemsep=0pt
\item[$(a)$] $N^k$ is a Nijenhuis operator on $T$ and hence $T\big[N^k\big]$ is a truss,
\item[$(b)$] $T\big[N^k\big]\big[N^l\big] = T\big[N^{k+l}\big]$,
\item[$(c)$] $N^l$ is a Nijenhuis operator on $T\big[N^k\big]$,
\end{itemize}
where $N^k$ means the $k$-fold composition of $N$, and $N^0=\id$.
\end{Theorem}
\begin{proof}
To simplify the notation we will write $\nij k$ for the product $\circ_{N^k}$ in $T\big[N^k\big]$. Obviously, $\nij 0 =\cdot$, the original multiplication in $T$.

First, we prove that for all $a,b\in T$, $k\in \N$,
\begin{subequations}\label{k-power}
\begin{align}
&N^k(a)N(b) = \big[N(N^k(a)b),N^{k+1}(ab),N^k(aN(b))\big],\label{ka}\\
&N(a)N^k(b) = \big[N(aN^k(b)),N^{k+1}(ab),N^k(N(a)b)\big],\label{kb}\\
&N^{k+1}(ab) = \big[N(N^k(a)b),N^k(a)N(b),N^k(aN(b))\big],\nonumber\\
&\hphantom{N^{k+1}(ab)}{} = \big[N^k(N(a)b),N(a)N^k(b) ,N\big(aN^k(b)\big)\big].\label{k+1}
\end{align}
\end{subequations}
We will prove equality \eqref{ka} by induction. The equality \eqref{kb} can be proven symmetrically, while \eqref{k+1} is an equivalent restatement of \eqref{ka} and \eqref{kb}; see Remark~\ref{rem.equal}.

For $k=0$, \eqref{ka} is automatically satisfied, while for $k=1$ this is the definition of a Nijenhuis operator. Assume that \eqref{ka} is true for $k$, then first, using the Nijenhuis condition and then the inductive assumption, we obtain
\begin{align*}
\begin{split}
 N^{k+1}(a)N(b) &= N\big(\big[N^{k+1}(a)b, N(N^k(a)b), N^k(a)N(b)\big]\big)\\
 &= N\big(\big[N^{k+1}(a)b, N(N^k(a)b), N(N^k(a)b),N^{k+1}(ab),N^k(aN(b))\big]\big)\\
 &= \big[N^{k+2}(a)b,N^{k+2}(ab),N^{k+1}(aN(b))\big].
\end{split}
\end{align*}
The final equality follows by the Mal'cev identity (the cancellation rule in the heap operation) and the fact that $N$ is a heap homomorphism. Therefore, \eqref{ka} is true for all natural $k$ by the principle of mathematical induction.

Using equations \eqref{k-power}, we can compute
\begin{align*}
 N(a)\nij k N(b) &= \big[N^{k+1}(a)N(b), N^{k}(N(a)N(b)), N(a)N^{k+1}(b)\big]\\
 &= \big[N(N^{k+1}(a)b),N^{k+2}(ab),N^{k+1}(aN(b)),N^{k+1}(aN(b)), N^{k+2}(ab)\\
 &\quad\ N^{k+1}(N(a)b),N\big(aN^{k+1}(b)\big),N^{k+2}(ab),N^{k+1}(N(a)b) \big]\\
 &= \big[N\big(N^{k+1}(a)b\big), N^{k+2}(ab),N\big(aN^{k+1}(b)\big) \big]
 = N(a\nij{k+1}b).
\end{align*}
The penultimate equality follows by the cancellation rules for an abelian heap operation. The last equality is a consequence of the fact that $N$ is a heap homomorphism and the definition of the Nijenhuis product. Starting with this, we can employ the inductive argument to prove that, for all $a,b\in T$, $k,l\in \N$,
\begin{equation}\label{nkl}
N^l(a\nij{k+l} b) = N^l(a)\nij k N^l(b).
\end{equation}
In particular, the case $k=0$ yields assertion (a).

The second assertion is also proved inductively on $l$. The inductive step is straightforward, so only the case $l=1$ needs to be proven. For all $a,b\in T$,
\begin{align*}
 [N(a)\nij k b, N(a\nij k b), a\nij k N(b)] &= \big[N^{k+1}(a)b, N^k(N(a)b), N(a)N^k(b),N(N^k(a)b)\\
 &\quad\ N^{k+1}(ab), N(aN^k(b)), N^k(a)N(b), N^k(aN(b)), aN^{k+1}(b)\big],\\
 &= \big[N^{k+1}(a)b, N^{k+1}(ab), N^{k+1}(ab),N^{k+1}(ab), aN^{k+1}(b)\big], \\
 &= \big[N^{k+1}(a)b,N^{k+1}(ab), aN^{k+1}(b)\big] = a\nij{k+1}b.
\end{align*}
The second equality follows by \eqref{k+1}. This completes the proof of statement (b). The last assertion follows immediately from (b) and \eqref{nkl}.
\end{proof}

Following \cite{4}, we propose:
\begin{Definition}\label{def.compat}
 Nijenhuis operators $N_1$, $N_2$ on a truss $T$ are said to be \textit{compatible} if, for all $a,b\in T$,
 \begin{gather*}
 N_1(a)N_2(b) = [ N_1(a\circ_{N_2}b),N_2(a)N_1(b), N_2(a\circ_{N_1}b) ].
 \end{gather*}
\end{Definition}
\begin{Example}\label{ex.com.id}
 The identity operator $\id$ on $T$ is compatible with any Nijenhuis operator on $T$.
\end{Example}
The next statement is the truss version of \cite[Theorems~3 and~4]{4}.
\begin{Theorem}\label{thm.comp}
 Let $T$ be a truss.
 \begin{itemize}\itemsep=0pt
 \item[$(1)$] If $N_1, N_2,\dots,N_{2n+1}$ are pairwise compatible Nijenhuis operators on $T$, then
 \[
 [N_1,N_2,\dots, N_{2n+1}]
 \]
 is a Nijenhuis operator on $T$.
 \item[$(2)$] For all Nijenhuis operators $N$ on $T$:
 \begin{itemize}\itemsep=0pt
 \item[$(a)$] the operators $N^k$ and $N^l$ are compatible, for all $k,l\in \N$,
 \item[$(b)$] $\big[N^{k_1},N^{k_2},\dots, N^{k_{2n+1}}\big]$ is a Nijenhuis operator for all $k_i,n\in \N$.
 \end{itemize}
 \end{itemize}
\end{Theorem}
\begin{proof}
 (1) We first note that, since the multiplication in the truss distributes over the heap operation and abelian heaps satisfy the rearrangement rules described in Remark~\ref{rem.rules}, for any heap homomorphisms $N_1, N_2,\dots ,N_{2n+1}\colon T\to T$ and all $a,b\in T$,
 \begin{gather}\label{n1n2n3}
 a\circ_{[N_1,N_2,\dots, N_{2n+1}]} b = [ a\circ_{N_1} b, a\circ_{N_2} b,\dots, a\circ_{N_{2n+1}} b].
 \end{gather}
 With \eqref{n1n2n3} at hand, we can prove the statement by induction on $n$. For $n=1$,
 \begin{align*}
 [N_1,N_2,N_3](a)\ &[N_1,N_2,N_3](b) = [N_1(a)N_1(b), N_1(a)N_2(b), N_1(a)N_3(b),
 N_2(a)N_1(b), \\ & N_2(a)N_2(b),
 N_2(a)N_3(b),
 N_3(a)N_1(b), N_3(a)N_2(b), N_3(a)N_3(b)]\\
 &= [N_1(a\circ_{N_1}b), N_1(a\circ_{N_2}b),N_2(a)N_1(b), N_2(a\circ_{N_1}b),
 N_1(a\circ_{N_3}b),\\
 &\quad\ N_3(a)N_1(b), N_3(a\circ_{N_1}b), N_2(a)N_1(b), N_2(a\circ_{N_2}b), N_2(a\circ_{N_3}b),\\
 &\quad\ N_3(a)N_2(b),N_3(a\circ_{N_2}b), N_3(a)N_1(b), N_3(a)N_2(b), N_3(a\circ_{N_3}b)]\\
 &= [N_1(a\circ_{N_1}b), N_1(a\circ_{N_2}b),N_1(a\circ_{N_3}b), N_2(a\circ_{N_1}b), N_2(a\circ_{N_2}b),
 \\
 &\quad\ N_2(a\circ_{N_3}b), N_3(a\circ_{N_1}b),N_3(a\circ_{N_2}b), N_3(a\circ_{N_3}b)]\\
 &= [N_1(a\circ_{[N_1,N_2,N_3]} b), N_2(a\circ_{[N_1,N_2,N_3]} b), N_3(a\circ_{[N_1,N_2,N_3]} b)]\\
 &= [N_1,N_2,N_3](a\circ_{[N_1,N_2,N_3]} b),
 \end{align*}
 where the definition of the heap bracket on operators and the truss distributive laws were used to derive the first equality, next the pairwise compatibility was employed. The third equality arises from the rearrangement and cancellation rules outlined in Remark~\ref{rem.rules}, while the next equality is a consequence of \eqref{n1n2n3}. Therefore, $[N_1,N_2,N_3]$ is a Nijenhuis operator as required.

 Next assume that the statement is true for $n=k-1$ and note that \eqref{n1n2n3} together with the rearrangement rules in Remark~\ref{rem.rules} imply that if $N_1,N_2,\dots, N_{2k+1}$ are pairwise compatible Nijenhuis operators then $N_{2k}$ and $N_{2k+1}$ are compatible with $N=[N_1,N_2,\dots, N_{2k-1}]$. Hence $[N, N_{2k},N_{2k+1}]$ is a Nijenhuis operator by the same arguments as those used above to establish the $n=1$ case.

 (2) Without any loss of generality, we may assume that $k\geq l$. As in the proof of Theorem~\ref{thm.Nij.iter}, we will write $\nij k$ for $\circ_{N^k}$, etc. In view of Theorem~\ref{thm.Nij.iter} and the definition of the Nijenhuis product~$\nij {k-l}$, we can compute, for all $a,b\in T$,
 \begin{align*}
 &\big[N^k(a\nij lb), N^l(a)N^k(b), N^l(a\nij k b)\big] = \big[N^{k-l}(N^l(a)N^l(b)),N^l(a)N^k(b), N^l(a)\nij {k-l} N^l(b) \big]\\
 &\quad= \big[N^{k-l}\big(N^l(a)N^l(b)\big),N^l(a)N^k(b), N^k(a)N^l(b), N^l(a)N^k(b), N^l(a)N^k(b)\big]\\
 &\quad= N^k(a)N^l(b).
 \end{align*}
The last equality follows by the cancellation and rearrangement rules recalled in Remark~\ref{rem.rules}. This completes the proof of statement (a). Statement (b) then follows by assertion (1).
\end{proof}

\section[Affine Nijenhuis operators and quantum bi-Hamiltonian systems]{Affine Nijenhuis operators and \\
quantum bi-Hamiltonian systems}\label{five}
In this section, first we apply the above discussion to trusses and operators arising from associative algebras and in this way extend the results of \cite{4} from the case of linear to affine maps. Next we construct an affine version of (weak) quantum bi-Hamiltonian systems. That is, we construct an affine Lie bracket (in the sense of \cite[Definition~1]{GraGra:Lie}) which can be represented as the commutator of a deformed associative bi-affine product on an affine space.

An associative algebra $A$ over a field $\F$ can be viewed as a truss with the original multiplication of the heap structure arising from the additive group, that is, $[a,b,c] = a-b+c$. To indicate this ternary point of view we write $\mathrm{T}(A)$. From this perspective an affine map $N\colon A\to A$ is a homomorphism of heaps that preserves affine or barycentric combinations, that is, for all $a,b,c\in A$ and $\lambda\in \F$,
\begin{align}\label{n.aff}
 N(a -b +c) = N(a)-N(b)+N(c), \qquad N((1-\lambda)a +\lambda b ) = (1-\lambda)N(a) + \lambda N(b).\!\!\!
\end{align}
The set of all affine maps $A\to A$ is denoted by $\mathrm{Aff}(A)$. This is a truss with the product given by composition and the heap operation defined pointwise. One easily checks that $\mathrm{Aff}(A)$ is an affine space over the vector space of all linear endomorphisms of $A$ with the operations defined pointwise.
\begin{Definition}\label{def.aff.Nij}
Let $A$ be an associative algebra and let $N\in \mathrm{Aff}(A)$. If $N$ is a Nijenhuis op\-er\-a\-tor on $\mathrm{T}(A)$ we refer to it as an \textit{affine Nijenhuis operator} on $A$.
\end{Definition}

Although, given $N\in \mathrm{Aff}(A)$ and $\lambda \in \F$, the function $\lambda N\colon A\to A$, $a\mapsto \lambda N(a)$, is not an affine map, it is still a homomorphism of heaps, i.e., the first of conditions \eqref{n.aff} is satisfied. Hence the following theorem can be stated.

\begin{Theorem}\label{thm.affine}
Let $A$ be an associative $\F$-algebra.
\begin{itemize}
 \item[$(1)$] If $N$ is an affine Nijenhuis operator on $A$, then for all $\lambda \in \F$, $\lambda N$ is a Nijenhuis operator on $\mathrm{T}(A)$ compatible with $N$.
\item[$(2)$] If $N_1,\dots , N_n\in \mathrm{Aff}(A)$ are pairwise compatible affine Nijenhuis operators on $\mathrm{T}(A)$, then, for all $\lambda_1,\dots, \lambda_n\in \F$ such that $\sum_{i=1}^n\lambda_i =1$,
 $N = \sum_{i=1}^n \lambda_i N_i $
 is an affine Nijenhuis operator on $A$.
\end{itemize}
\end{Theorem}
\begin{proof}
 (1) First, note that for all $a,b\in A$,
 \begin{align}\label{linear}
 a\circ_{\lambda N}b = \lambda N(a)b - \lambda N(ab) + a\lambda N(b) = \lambda a\circ _N b.
 \end{align}
 Hence, if $N$ is a Nijenhuis operator on $\mathrm{T}(A)$, then
 \[
 \lambda N(a\circ_{\lambda N}b) = \lambda^2 N(a)N(b) = \lambda N(a)\lambda N(b)
 \]
 as required.
The compatibility property likewise follows by \eqref{linear}.

(2) The correspondence between Nijenhuis products in \eqref{linear} implies that the Nijenhuis operators $\lambda_iN_i$, $i=1,\dots n$ are pairwise compatible. Since $\sum_{i=1}^n\lambda_i =1$,
\begin{align*}
 N &= \sum_{i=1}^n \lambda_i N_i = N_1 - \lambda_2 N_1 + \lambda_2 N_2 - \lambda_3 N_1 + \dots - \lambda_{n} N_1 + \lambda_{n}N_{n}\\
 &= [N_1 ,\lambda_2 N_1 , \lambda_2 N_2 , \lambda_3 N_1 , \dots , \lambda_{n} N_1 , \lambda_{n}N_{n}],
\end{align*}
and hence the affine map $\sum_{i=1}^n \lambda_i N_i$ is a Nijenhuis operator on $\mathrm{T}(A)$ and hence an affine Nijenhuis operator on $A$ by assertion (1) in Theorem~\ref{thm.comp}.
\end{proof}

\begin{Corollary}\label{cor.aff.nij.op}
 If $P$ is a multiplicative idempotent in $\mathrm{Aff}(A)$, then for all $\alpha\in \F$, $(1-\alpha)P + \alpha\, \id$ is an affine Nijenhuis operator.
\end{Corollary}
\begin{proof}
 This follows immediately from Theorem~\ref{thm.affine} and Example~\ref{ex.com.id}.
\end{proof}

Let $A$ be an affine space over an $\F$-vector space $\vec{A}$. As explained for example in \cite[Section~4]{BreBrz:hea} or \cite{Brz:Lie},
the action $+$ of $\vec{A}$ on $A$ makes the latter an abelian heap with the operation given by, for all $a,b,c \in A$,
\begin{equation}\label{aff.heap}
 [a,b,c] = a + \overrightarrow{bc},
\end{equation}
where $\overrightarrow{bc}$ is the unique vector in $\vec{A}$ from $b$ to $c$, i.e., such that $c = b+ \overrightarrow{bc}$.
With this interpretation, an affine map from $B$ over $\vec{B}$ to $A$ corresponds to a heap homomorphism $f\colon A\to B$ such that, for all $a,b\in A$ and $\lambda \in \F$,
\begin{equation}\label{aff.map}
 f\big(a + \lambda \overrightarrow{ab}\big) = f(a) + \lambda \overrightarrow{f(a)f(b)}.
\end{equation}
Any such map defines uniquely linear transformation $\vec{f}\colon \vec{A} \to \vec{B}$ by $\overrightarrow{ab}\mapsto \overrightarrow{f(a)f(b)}$. We refer to it as a \textit{linearisation} of $f$.

If $A$ is a vector space, then it is an affine space over itself with the vector from $a$ to $b$ being simply the difference $b-a$. The heap operation \eqref{aff.heap} coincides then with $a-b+c$, while to be affine map from $A$ to $A$ in the sense of \eqref{aff.map} is equivalent to satisfying conditions \eqref{n.aff}.

Recall from \cite[Definition~1]{GraGra:Lie} that a Lie bracket on an affine space $A$ is an anti-symmetric bi-affine map $[-,-]\colon A\times A\to \vec{A}$ satisfying the Jacobi identity
\[
\overrightarrow{[[a,b],c]} + \overrightarrow{[[b,c],a]} +\overrightarrow{[[c,a],b]} =0.
\]
The arrows over the brackets indicate the linearisations of affine maps $[-,b]\colon A\to \vec{A}$.

Let $A$ be an affine space with a bi-affine associative multiplication $\cdot\colon A\times A \to A$ (we will keep writing the dot between the elements of $A$ in order to avoid the confusion with the end points of the vector in $\vec{A}$). The fact that, for all $a\in A$, the function $A\to A$, $b\mapsto a\cdot b$ is an affine map implies in particular that it is a heap homomorphism which is equivalent to say that the multiplication left-distributes over the heap operation \eqref{aff.heap}. Similarly, the heap homomorphism property of maps $b\mapsto b\cdot a$ yield the right truss distributive law. In short, $A$ is a truss, which might be called an \textit{affine truss} or an associative \textit{affgebra} -- the term coined in \cite{GraGra:Lie}.
\begin{Remark}\label{rem.ideal.aff}
In the same way as a truss can be embedded in a ring (see Remark~\ref{rem.ideal}) any associative affgebra can be obtained as a coset in an associative algebra. Explicitly, given an algebra $A$, an ideal $I$ of $A$ and an idempotent element $q\in A$, $T(I;q) = q+I$ is an affine space over $I$ with $\overrightarrow{(q +x)(q+y)}= y-x$, to which the multiplication on $A$ restricts as a bi-affine map.
\end{Remark}

With no additional effort, the notion of an affine Nijenhuis operation and the statement (2) of Theorem~\ref{thm.affine} can be extended to affgebras.

\begin{Proposition}\label{prop.affine}
 If $N_1,\dots , N_n\in \mathrm{Aff}(A)$ are pairwise compatible affine Nijenhuis operators on an associative $\F$-affgebra $A$, then, for all $\lambda_1,\dots, \lambda_n\in \F$ such that $\sum_{i=1}^n\lambda_i =1$,
 $N = \sum_{i=1}^n \lambda_i N_i $
 is an affine Nijenhuis operator on $A$.
\end{Proposition}

\begin{Example}
 Let $A$ be an associative algebra, $I$ be an ideal in $A$ and $q\in A$ and idempotent element. Assume that $I$ decomposes into a sum of two ideals in $A$, $I=I_1\oplus I_2$. Let $P_i\colon I\to I_i$, $i=1,2$ be corresponding projections, such that $P_i(xq) = P_i(x)q$ and $P_i(qx) = qP_i(x)$, for all $x\in I$, $i=1,2$. By \cite[Theorem~5]{4}, for all $\lambda_1,\lambda_2 \in \F$, $\lambda_1P_1+\lambda_2P_2$ is a Nijenhuis operator on $I$, in particular each of the $P_i$ is a Nijenhuis operator. In view of Proposition~\ref{prop.Nij.ideal},
 \[
 N\colon\ T(q;I) \to T(q;I), \qquad q+x \mapsto q + \lambda_1 P_1(x) + \lambda_2 P_2(x)
 \]
 is an Nijenhuis operator on $T(q;I)$. On the other hand, $N$ can be understood as an affine combination of operators $N_i(q+x) = q +P_i(x)$ and $Q(q+x) = q$ as
 \[N= (1-\lambda_1-\lambda_2)\,Q + \lambda_1 N_1 +\lambda_2 N_2.
 \]

 For an explicit example, we can take the algebra $R$, its ideal $I$ and an idempotent $q$ described in Example~\ref{ex.proj}. Every element of $I$ can be uniquely decomposed into the sum of an upper triangular and strictly lower triangular matrix. If $P_-$ denotes the projection on the latter, for all scalars $\lambda_1$, $\lambda_2$, we obtain the following affine Nijenhuis operator on $T(I;q)$:
 \[
 N(A(\mathbf{a},\mathbf{b},1)) = A(\lambda_1P_+(\mathbf{a})+\lambda_2P_-(\mathbf{a}),\lambda_1\mathbf{b},1).
 \]
\end{Example}

Any associative affgebra $A$ admits a Lie bracket given by the linearised commutator, for all $a,b\in A$,
\begin{equation}\label{affine.com}
 [a,b] =\overrightarrow{(b\cdot a)(a\cdot b)}.
\end{equation}
Indeed, $[a,b]$ is clearly anti-symmetric, and, for all $a,b,c\in A$,
\begin{align*}
 \overrightarrow{[[a,b],c]} + \mathrm{cycl.} &= \overrightarrow{\left[\overrightarrow{(b\cdot a)(a\cdot b)},c\right]} +\mathrm{cycl.}
 = [a\cdot b,c] - [b\cdot a,c] + \mathrm{cycl.}\\ &= \overrightarrow{(c\cdot a\cdot b)(a\cdot b\cdot c)} - \overrightarrow{(c\cdot b\cdot a)(b\cdot a\cdot c)} + \mathrm{cycl.} =0.
\end{align*}
In the case of the affgebra $T(I;q)$ of Remark~\ref{rem.ideal.aff}, the Lie bracket comes out as the translation of the standard commutator, i.e., $[q+x,q+y]= q+[x,y]$.

With all these preliminaries at hand we can state the following affine version of \cite[Theorem~8]{4}.

\begin{Theorem}\label{thm.affine.Lie}
Let $N$ be an affine Nijenhuis operator on an associative affgebra $A$. Let $[-,-]$ be the Lie bracket \eqref{affine.com}.
\begin{itemize}\itemsep=0pt
 \item[$(1)$] The multiplication $\circ_N$ is a bi-affine operation, thus making $A[N]$ into an associative affgebra.
 \item[$(2)$] The operation $[-,-]_N\colon A\times A\to \vec{A}$ given by
 \[
 [a,b]_N := [N(a),b] -\vec{N}([a,b]) + [a, N(b)],
 \]
 for all $a,b\in A$, is a Lie bracket on $A$ such that
 \[
 [a,b]_N = \overrightarrow{(b\circ_N a)(a\circ_N b)}.
 \]
 \item[$(3)$] For all $a,b\in A$,
 \[
\vec{N}([a,b]_N) = [N(a), N(b)].
 \]
\end{itemize}
\end{Theorem}
\begin{proof}
First, note that in view of the definition of the heap operation \eqref{aff.heap}, the Nijenhuis product comes out as
\[
a\circ_Nc = N(a)\cdot c + \overrightarrow{N(a\cdot c)(a\cdot N(c))}.
\]
 We will use repeatedly the following elementary facts from the theory of affine spaces. For all points $a,b,c,d\in A$ and all vectors $v,w\in \vec{A}$,
 \begin{gather}\label{aff.vect}
 \overrightarrow{(a+v)(b+w)} = \overrightarrow{ab}-v +w, \qquad \overrightarrow{cd} - \overrightarrow{ab} = \overrightarrow{bd} - \overrightarrow{ac} .
 \end{gather}

To check if $\circ_N$ is a bi-affine multiplication, take any $a,b,c\in A$ and $\lambda\in \F$, and using the facts that $N$ is an affine map, the multiplication $\cdot$ is bi-affine and \eqref{aff.vect} compute
\begin{align*}
 \bigl(a+\lambda \overrightarrow{ab}\bigr)\circ_N  c &= \bigl(N(a) + \lambda\overrightarrow{N(a)N(b)}\bigr)\cdot c
 \\
 &\quad+\overrightarrow{N\big(a\cdot c+\lambda\overrightarrow{(a\cdot c)(b\cdot c)}\big)\big((a+\lambda \overrightarrow{ab})\cdot N(c)\big)}\\
 &= N(a)\cdot c + \lambda \overrightarrow{(N(a)\cdot c)(N(b)\cdot c)} \\
 &\quad
 +
 \overrightarrow{\big(N(a\cdot c)+\lambda \overrightarrow{N(a\cdot c)N(b\cdot c)}\big)\big(a\cdot N(c)+\lambda \overrightarrow{(a\cdot N(c))(b\cdot N(c)})\big)}\\
 &= N(a)\cdot c + \lambda \overrightarrow{(N(a)\cdot c)(N(b)\cdot c)} + \overrightarrow{N(a\cdot c)(a\cdot N(c))}\\
 &\quad - \lambda \overrightarrow{N(a\cdot c)N(b\cdot c)} +
 \lambda \overrightarrow{(a\cdot N(c))(b\cdot N(c))}\\
 &= a\circ_N c+ \lambda \big(\overrightarrow{(N(a)\cdot c)(N(b)\cdot c)} - \overrightarrow{N(a\cdot c)(a\cdot N(c))} +
 \overrightarrow{N(b\cdot c)(b\cdot N(c))}\big)\\
 &= a\circ_N c + \lambda \overrightarrow{\big(N(a)\cdot c+ \overrightarrow{ N(a\cdot c)(a\cdot N(c))} \big)\big(N(b)\cdot c+
 \overrightarrow{N(b\cdot c)(b\cdot N(c))}\big)}\\
 &= a\circ_N c + \lambda \overrightarrow{(a\circ_N c)(b\circ_N c)}
\end{align*}
as required. The second compatibility condition is proven in a symmetric way. Therefore, $A[N]$~is an associative affgebra.

Using properties \eqref{aff.vect}, we find, for all $a,b\in A$,
\begin{align*}
 \overrightarrow{(b\circ_N a)(a\circ_N b)}
 &= \overrightarrow{\left(N(b)\cdot a + \overrightarrow{N(b\cdot a)(b\cdot N(a))}\right)\left(N(a)\cdot b + \overrightarrow{N(a\cdot b)(a\cdot N(b))}\right)}\\
 &= \overrightarrow{(N(b)\cdot a)(N(a)\cdot b)}-\overrightarrow{N(b\cdot a)(b\cdot N(a))}+ \overrightarrow{N(a\cdot b)(a\cdot N(b))}\\
 &= \overrightarrow{(N(b)\cdot a)(a\cdot N(b))}-\overrightarrow{N(b\cdot a)N(a\cdot b)}+ \overrightarrow{(b\cdot N(a))(N(a)\cdot b)}\\
 &=[N(a),b] -\vec{N}([a,b]) + [a, N(b)] =[a,b]_N.
\end{align*}
In view of the fact that the Nijenhuis product $\circ_N$ makes $A$ an associative affgebra, this proves both assertions in statement (2).

Finally, since $N$ is an affine map and a Nijenhuis operator on $A$,
\begin{align*}
\vec{N}([a,b]_N) &=\vec{N}\left( \overrightarrow{(b\circ_N a)(a\circ_N b)}\right) =\overrightarrow{N(b\circ_N a)N(a\circ_N b)}\\
 &= \overrightarrow{\left(N(b)\cdot N(a)\right)\left(N(a)\cdot N (b)\right)} = [N(a),N(b)].
\end{align*}
This completes the proof of the theorem.
\end{proof}

\subsection*{Acknowledgements}
The research of Tomasz Brzezi\'nski is partially supported by the National Science Centre, Poland, grant no. 2019/35/B/ST1/01115.

\pdfbookmark[1]{References}{ref}
\LastPageEnding


\begin{thebibliography}{99}
\footnotesize\itemsep=0pt

\bibitem{AndBrzRyb:ext}
Andruszkiewicz R.R., Brzezi\'nski T., Rybo{\l}owicz B., Ideal ring extensions
 and trusses, \href{https://doi.org/10.1016/j.jalgebra.2022.01.038}{\textit{J.~Algebra}} \textbf{600} (2022), 237--278,
 \href{https://arxiv.org/abs/2101.09484}{arXiv:2101.09484}.

\bibitem{5}
Baer R., Zur {E}inf\"uhrung des {S}charbegriffs, \href{https://doi.org/10.1515/crll.1929.160.199}{\textit{J.~Reine Angew. Math.}}
 \textbf{160} (1929), 199--207.

\bibitem{Ben:fib}
Benenti S., Fibr\'es affines canoniques et m\'ecanique newtonienne, in Action
 hamiltoniennes de groupes. {T}roisi\`eme th\'eor\`eme de {L}ie ({L}yon,
 1986), \textit{Travaux en Cours}, Vol.~27, Hermann, Paris, 1988, 13--37.

\bibitem{BreBrz:hea}
Breaz S., Brzezi\'nski T., Rybo{\l}owicz B., Saracco P., Heaps of modules and
 affine spaces, \textit{Ann. Mat. Pura Appl.}, {t}o appear, \href{https://arxiv.org/abs/2203.07268}{arXiv:2203.07268}.

\bibitem{3}
Brzezi\'nski T., Trusses: between braces and rings, \href{https://doi.org/10.1090/tran/7705}{\textit{Trans. Amer. Math.
 Soc.}} \textbf{372} (2019), 4149--4176, \href{https://arxiv.org/abs/1710.02870}{arXiv:1710.02870}.

\bibitem{Brz:par}
Brzezi\'nski T., Trusses: paragons, ideals and modules, \href{https://doi.org/10.1016/j.jpaa.2019.106258}{\textit{J.~Pure Appl.
 Algebra}} \textbf{224} (2020), 106258, 39~pages, \href{https://arxiv.org/abs/1901.07033}{arXiv:1901.07033}.

\bibitem{Brz:Lie}
Brzezi\'nski T., Lie trusses and heaps of Lie affebras, \href{https://doi.org/10.22323/1.406.0307}{\textit{Proc. Sci.}}
 \textbf{406} (2022), PoS(CORFU2021)307, 12~pages, \href{https://arxiv.org/abs/2203.12975}{arXiv:2203.12975}.

\bibitem{4}
Cari\~nena J.F., Grabowski J., Marmo G., Quantum bi-{H}amiltonian systems,
 \href{https://doi.org/10.1142/S0217751X00001954}{\textit{Internat.~J.~Modern Phys.~A}} \textbf{15} (2000), 4797--4810,
 \href{https://arxiv.org/abs/math-ph/0610011}{arXiv:math-ph/0610011}.

\bibitem{2}
Certaine J., The ternary operation {$(abc)=ab^{-1}c$} of a group, \href{https://doi.org/10.1090/S0002-9904-1943-08042-1}{\textit{Bull.
 Amer. Math. Soc.}} \textbf{49} (1943), 869--877.

\bibitem{7}
Gerstenhaber M., On the deformation of rings and algebras, \href{https://doi.org/10.2307/1970484}{\textit{Ann. of
 Math.}} \textbf{79} (1964), 59--103.

\bibitem{GraGra:Lie}
Grabowska K., Grabowski J., Urba\'nski P., Lie brackets on affine bundles,
 \href{https://doi.org/10.1023/A:1024457728027}{\textit{Ann. Global Anal. Geom.}} \textbf{24} (2003), 101--130,
 \href{https://arxiv.org/abs/math.DG/0203112}{arXiv:math.DG/0203112}.

\bibitem{Hoch}
Hochschild G., On the cohomology groups of an associative algebra, \href{https://doi.org/10.2307/1969145}{\textit{Ann.
 of Math.}} \textbf{46} (1945), 58--67.

\bibitem{Kos:Nij}
Kosmann-Schwarzbach Y., Nijenhuis structures on {C}ourant algebroids,
 \href{https://doi.org/10.1007/s00574-011-0032-5}{\textit{Bull. Braz. Math. Soc.~(N.S.)}} \textbf{42} (2011), 625--649,
 \href{https://arxiv.org/abs/1102.1410}{arXiv:1102.1410}.

\bibitem{Mag:sim}
Magri F., A simple model of the integrable {H}amiltonian equation,
 \href{https://doi.org/10.1063/1.523777}{\textit{J.~Math. Phys.}} \textbf{19} (1978), 1156--1162.

\bibitem{MasVig:non}
Massa E., Vignolo S., Bruno D., Non-holonomic {L}agrangian and {H}amiltonian
 mechanics: an intrinsic approach, \href{https://doi.org/10.1088/0305-4470/35/31/313}{\textit{J.~Phys.~A}} \textbf{35} (2002),
 6713--6742.

\bibitem{6}
Pr\"ufer H., Theorie der {A}belschen {G}ruppen, \href{https://doi.org/10.1007/BF01188079}{\textit{Math.~Z.}} \textbf{20}
 (1924), 165--187.

\bibitem{Tul:fra}
Tulczyjew W.M., Frame independence of analytical mechanics, \textit{Atti Accad.
 Sci. Torino Cl. Sci. Fis. Mat. Natur.} \textbf{119} (1985), 273--279.

\bibitem{Urb:aff}
Urba\'nski P., Affine {P}oisson structures in analytical mechanics, in
 Quantization and Infinite-Dimensional Systems ({B}ialowieza, 1993), \href{https://doi.org/10.1007/978-1-4615-2564-6_15}{Plenum},
 New York, 1994, 123--129.

\bibitem{Wei:uni}
Weinstein A., A universal phase space for particles in {Y}ang--{M}ills fields,
 \href{https://doi.org/10.1007/BF00400169}{\textit{Lett. Math. Phys.}} \textbf{2} (1978), 417--420.

\end{thebibliography}
\end{document}